\documentclass[reqno,a4paper,11pt]{amsart}
\usepackage[T1]{fontenc}
\usepackage[utf8]{inputenc}
\usepackage[english]{babel}
\usepackage{amsmath}
\usepackage{amssymb}
\usepackage{amsthm,thmtools}
\usepackage{setspace}
\usepackage{graphicx}
\usepackage{geometry}
\usepackage{tikz-cd}

\tikzcdset{arrow style=tikz, diagrams={>=Straight Barb}}

\usepackage{xcolor}
\usepackage{bbm}
\usepackage{mdframed}

\usepackage{libertine}
\usepackage[libertine]{newtxmath}

\makeatletter 
\def\l@subsection{\@tocline{2}{0pt}{1pc}{5pc}{}} \def\l@subsection{\@tocline{2}{0pt}{2pc}{6pc}{}}
\makeatother

\declaretheoremstyle[headfont=\bfseries,bodyfont=\itshape]{myplain}
\declaretheoremstyle[headfont=\bfseries,bodyfont=\normalfont]{mydefinition}
\declaretheoremstyle[headfont=\bfseries,bodyfont=\normalfont,qed=$\diamond$]{myexample}
\declaretheoremstyle[headfont=\itshape,bodyfont=\normalfont, qed=$\diamond$]{myremark}

\declaretheorem[style=myplain, sharenumber=theo, name=Lemma]{lemma}
\declaretheorem[style=myplain, sharenumber=theo, name=Corollary]{coroll}
\declaretheorem[style=myplain, sharenumber=theo, name=Proposition]{prop}
\declaretheorem[style=myplain, sharenumber=theo, name=Proposition/Definition]{prop/def}

\declaretheorem[style=mydefinition, sharenumber=theo, name=Definition]{definition}

\newtheorem*{theorem*}{Theorem}

\declaretheorem[style=myexample, sharenumber=theo, name=Example]{example}
\declaretheorem[style=myremark, sharenumber=theo, name=Remark]{rem}

\numberwithin{equation}{section}

\usepackage[colorlinks=true]{hyperref}
\hypersetup{linkcolor=blue, citecolor=blue, filecolor=black, urlcolor=blue}
\everymath{\displaystyle}

\begin{document}
	\title{Basic sections of $\mathcal{LA}$-groupoids}
	
	\author{Antonio Maglio}
	\address{Institute of Mathematics, Polish Academy of Sciences, \'Sniadeckich 8, 00-656 Warszawa, Poland}
	\curraddr{}
	\email{\href{mailto:amaglio@impan.pl}{amaglio@impan.pl}}
	\thanks{}
	
	\author{Fabricio Valencia}
	\address{Instituto de Matem\'atica e Estat\'istica, Universidade de S\~ao Paulo, Rua do Mat\~ao 1010, Cidade Universit\'aria, 05508-090 S\~ao Paulo, Brazil}
	\email{\href{mailto:fabricio.valencia@ime.usp.br}{fabricio.valencia@ime.usp.br}}
	
\begin{abstract}
We define the notion of basic section of an $\mathcal{LA}$-groupoid whose core-anchor map is injective. Such a notion turns out to be Morita invariant, so that it provides a simpler model for the sections of the stacky Lie algebroids presented by such $\mathcal{LA}$-groupoids, yet equivalent to the well-known model provided by their multiplicative sections.
\end{abstract}
	
	\maketitle
	
	\section{Introduction}
	 Lie groupoids constitute a framework which has received much attention in recent years since they generalize manifolds, Lie groups, Lie group actions, submersions, foliations, pseudogroups, vector bundles, and principal bundles, among others, thus providing a new perspective on classical geometric questions and results. These geometric objects can be viewed as an intermediate step in defining differentiable stacks, which are spaces admitting singularities and generalizing manifolds, orbifolds and leaf spaces of singular foliations. Recently, Lie groupoids equipped with geometric structures suitably compatible with Morita equivalence have been the object of intense research. This is due to the fact that such structures descend to the quotient stack of a Lie groupoid, giving rise to the extension of several geometric notions over singular orbit spaces. For instance, the notion of Morita equivalence of $VB$-groupoids plays a role in defining vector bundles over differentiable stacks \cite{dHO20}. Additionally, the space of multiplicative sections of an $\mathcal{LA}$-groupoid has the structure of a Lie 2-algebra which is Morita invariant, so that the space of vector fields on a differentiable stack has attached a natural structure of Lie $2$-algebra \cite{OW19}. In particular, if we restrict our attention to the case of differentiable stacks presented by foliation groupoids then there is a simpler model for the stacky vector fields which is constructed in terms of a notion of basic vector field \cite{HS21}. The injectivity assumption imposed on the anchor map of a foliation groupoid turns out to be closely related to the notion of $0$-shifted symplectic structure. In this regard, basic vector fields have been used to describe a Morita invariant reduction of Lie groupoids equipped with $0$-shifted symplectic structures under Hamiltonian actions of foliation Lie $2$-groups \cite{HS21}.

Building upon the idea of basic vector field on a foliation groupoid, in this short paper we generalize the related constructions from \cite{HS21} regarding this notion to study the space of sections of an $\mathcal{LA}$-groupoid whose core-anchor map is injective. Surprisingly, we can introduce a simpler approach to describe basic sections over this sort of $\mathcal{LA}$-groupoids by applying some of the general machinery developed in \cite{dHO20,MTV24,OW19}.

Our main result can be summarized in the following way.

\begin{theorem*}
	Let $(V\rightrightarrows E; G\rightrightarrows M)$ be an $\mathcal{LA}$-groupoid with core $C$ whose core-anchor map $\partial\colon C\to E$ is injective. Then, there exists a Lie algebra structure on the space of basic sections on $E/\partial(C)$ which is invariant under Morita equivalence. Furthermore, such a Lie algebra is quasi-isomorphic to the Lie $2$-algebra of multiplicative sections of $V$. 
\end{theorem*}

It is important to stress that this result also applies to describe basic sections of other interesting constructions where multiplicative or basic geometric structures over Lie groupoids yield associated $\mathcal{LA}$-groupoids. Namely, our initial motivation comes from trying to define what would be a basic derivation of a line bundle groupoid, an object that is closely related to the notion of $0$-shifted contact structure recently introduced in \cite{MTV24}. This is because those basic derivations remarkably allow us to obtain a reduction procedure for line bundle groupoids equipped with $0$-shifted contact structures invariant under the action of foliation Lie $2$-groups \cite{MV}.

The work is structured as follows. In Section \ref{sec:2}, we briefly introduce the fundamental facts about $VB$-groupoids that we shall be using throughout our study, focusing on defining the Lie 2-algebra of multiplicative sections associated with any $\mathcal{LA}$-groupoid. In Section \ref{sec:3}, we describe the notion of basic sections of a $VB$-groupoid whose core-anchor map is injective, proving that, in the case of an $\mathcal{LA}$-groupoid, the space of all basic sections inherits a natural Lie algebra structure. Such a Lie algebra is Morita invariant, so that it can be thought of as the space of sections of the stacky Lie algebroid presented by this kind of  $\mathcal{LA}$-groupoid. We also show that our approach provides a model for stacky sections which is still equivalent to the well-known model provided by multiplicative sections of $\mathcal{LA}$-groupoids. Finally, in Section \ref{sec:4}, we illustrate our constructions with several interesting examples where the main result of this work applies. Appendix \ref{appendix} contains some basic facts about Lie algebroids which are used throughout the sections.

	\vspace{.2cm}
	{\bf Acknowledgments:}  We are very thankful to Cristián Ortiz and Luca Vitagliano for enlightening discussion. We are indebted to the anonymous referee because all her/his corrections, comments, and suggestions greatly improved this work. The research of A.~M.~was funded by the National Science Centre (Poland) within the project WEAVE-UNISONO, No. 2023/05/Y/ST1/00043. A.~M.~was also partially supported by GNSAGA of INdAM. F.~V.~was supported by Grant 2024/14883-6 S\~ao Paulo Research Foundation - FAPESP.
	
	\section[\texorpdfstring{$\mathcal{LA}$-groupoids}{LA-groupoids}]{\texorpdfstring{$\mathcal{LA}$-groupoids and their multiplicative sections}{LA-groupoids}}\label{sec:2}
	
In this short section we briefly introduce the elementary notions and constructions that we will be using throughout the work. The reader is recommended to consult \cite{GSM17,OW19} for specific details.

We think of a \emph{$VB$-groupoid} as a vector bundle internal to the category of Lie groupoids. That is to say, a $VB$-groupoid is a commutative diagram
	\begin{equation}
		\label{eq:VBG}
		\begin{tikzcd}
			V \arrow[r, shift left=0.5ex] \arrow[r, shift right= 0.5ex] \arrow[d] & E \arrow[d]\\
			G \arrow[r, shift left=0.5ex] \arrow[r, shift right=0.5ex] & M
		\end{tikzcd}
	\end{equation}
	where $V\rightrightarrows E$ and $G\rightrightarrows M$ are Lie groupoids, $V\to G$ and $E\to M$ are vector bundles ($VB$ in what follows), and all the Lie groupoid structure maps of $V$ are $VB$-morphisms covering the Lie groupoid structure maps of $G$. We denote by $s,t,u, m$ and $i$ the source, target, unit, composition and inversion maps of $G$ and by $\tilde{s}, \tilde{t}, \tilde{u}, \tilde{m}$ and $\tilde{i}$ the corresponding ones of $V$. Moreover, we denote the $VB$-groupoid \eqref{eq:VBG} as the pair $(V\rightrightarrows E; G\rightrightarrows M)$ or only as $V$ if there is not risk of confusion. The \emph{core} $C$ of a $VB$-groupoid $V$ is the vector bundle $C$ over $M$ given by
	\[
		C= \ker(\tilde{s})|_M.
	\]
	
	The \emph{core-anchor map} of $V$ is the $VB$-morphism given by the restriction of the target $\tilde{t}$ to the core $C$. This will be denoted by $\partial \colon C\to E$. Accordingly, the \emph{core complex} of $V$ is the $2$-term complex given by the core-anchor map:
	\[
		\begin{tikzcd}
			0 \arrow[r] & C \arrow[r, "\partial"] & E \arrow[r] & 0.
		\end{tikzcd}
	\]
	
	Any section $c$ of the core $C$ of a $VB$-groupoid $V$ determines a \emph{right invariant section} $\overrightarrow{c}$ and a \emph{left invariant section} $\overleftarrow{c}$ of $V$ by setting
	\[
		\overrightarrow{c}_g= c_{t(g)}\cdot 0_g \in V_g \quad \text{and}\quad \overleftarrow{c}_g= -0_g\cdot c_{s(g)}^{-1}, \quad g\in G.
	\]

	A \emph{$VB$-groupoid morphism} is a pair $(F,f)\colon (V\rightrightarrows E; G\rightrightarrows M)\to (V'\rightrightarrows E'; G'\rightrightarrows M')$, where $F\colon (V\rightrightarrows E)\to (V'\rightrightarrows E')$ and $f\colon (G\rightrightarrows M)\to (G'\rightrightarrows M')$ are Lie groupoid morphisms, and $(F,f)\colon (V\to G)\to (V'\to G')$ and $(F,f)\colon (E\to M)\to (E'\to M')$ are vector bundle morphisms. In these terms, a \emph{$VB$-Morita map} is defined to be a $VB$-groupoid morphism $(F,f)$ such that $F$ is also a Morita map, see \cite{dHO20}. Besides, a \emph{linear natural transformation} between two $VB$-groupoid morphisms $F,K\colon V\to V'$ is a $VB$-morphism $\alpha\colon E\to V'$ that is also a natural transformation, compare \cite{dHO20,MTV24}.
	
	\begin{definition}
	An \emph{$\mathcal{LA}$-groupoid} is a $VB$-groupoid $(V\rightrightarrows E; G\rightrightarrows M)$, where both $V$ and $E$ are Lie algebroids over $G$ and $M$, respectively, and all the structure maps of $V$ are Lie algebroid morphisms. Additionally, 	an \emph{$\mathcal{LA}$-groupoid morphism} is a $VB$-groupoid morphism $(F,f)$ between two $\mathcal{LA}$-groupoids $V$ and $V'$ such that $(F,f)\colon (V\to G)\to (V'\to G')$ and $(F,f)\colon (E\to M)\to (E'\to M')$ are Lie algebroid morphisms. 
	\end{definition}
	
	Let $V$ be an $\mathcal{LA}$-groupoid. In this case, the core $C$ of $V$ possesses a Lie algebroid structure whose anchor map is determined by the composition of the core-anchor map $\partial\colon C\to E$ with the anchor $\rho_E\colon E\to TM$ of $E$, and the bracket $[-,-]_C$ maps two sections $c,c'\in \Gamma(C)$ into the unique section of $C$ such that $\overrightarrow{[c,c']_C}= [\overrightarrow{c}, \overrightarrow{c'}]_V\in \Gamma(V)$. Such an algebroid structure is defined in such a way that the core-anchor map $\partial$ becomes a Lie algebroid morphism. Also, an $\mathcal{LA}$-groupoid morphism $(F,f)$ from $V$ to $V'$ determines a Lie algebroid morphism from the core $C$ of $V$ to the core $C'$ of $V'$.
	
An \emph{$\mathcal{LA}$-Morita map} is defined to be an $\mathcal{LA}$-groupoid morphism $(F,f)$ such that $F$ is also a Morita map, see \cite{OW19}. The latter concept induces a notion of Morita equivalence of $\mathcal{LA}$-groupoids by saying that two $\mathcal{LA}$-groupoids $V$ and $V'$ are \emph{Morita equivalent} if there exists a third $\mathcal{LA}$-groupoid $W$ and two $\mathcal{LA}$-Morita maps $W\to V$ and $W\to V'$. This gives rise to an equivalence relation which, in turn, allows us to speak about a notion of \emph{stacky Lie algebroid} by considering the equivalence class of an $\mathcal{LA}$-groupoid up to Morita equivalence. See \cite{Wa15} for specific details regarding the notion of Lie algebroid over a differentiable stack.

We are interested in looking at the Lie 2-algebra of multiplicative sections of an $\mathcal{LA}$-groupoid, that is the one providing a model for the sections of a stacky Lie algebroid \cite{OW19}. A \emph{multiplicative section} of an $\mathcal{LA}$-groupoid $(V\rightrightarrows E; G\rightrightarrows M)$ is a pair $(\lambda, e)\in \Gamma(V)\oplus\Gamma(E)$ such that $\lambda \colon G\to V$ is a Lie groupoid morphism covering $e\colon M\to E$. We denote by $\Gamma_{\operatorname{mult}}(V)$ the space of multiplicative sections of $V$. This inherits a Lie algebra structure out of the $\mathcal{LA}$-groupoid structure of $V$. As shown in \cite{OW19}, there exists a Lie 2-algebra of multiplicative sections on $V$ which can be codified in terms of the following crossed module of Lie algebras. If $ \partial\colon C \to E$ is the core-anchor map of $V$, then the crossed module of multiplicative sections of $V$ is given by
		\[
			\begin{tikzcd}
				\Gamma(C)\arrow[r, "\delta"] & \Gamma_{\operatorname{mult}}(V)\arrow[r, "D"] & \operatorname{Der}(\Gamma(C)),
			\end{tikzcd}
		\]
		where $\delta$ maps a section $c\in \Gamma(C)$ to the pair $(\overrightarrow{c}-\overleftarrow{c}, \partial(c))\in \Gamma(V)\oplus \Gamma(E)$ and $D$ maps the pair $(\lambda, e)$ to the derivation $D_{(\lambda, e)}$ defined by $D_{(\lambda, e)}(c)=
			[\overrightarrow{c}, \lambda]|_M \in \Gamma(C)$ for all $c\in \Gamma(C)$. Remarkably, such a crossed module of Lie algebras is Morita invariant, meaning that $\mathcal{LA}$-groupoids related by an $\mathcal{LA}$-Morita map determine quasi-isomorphic crossed modules of multiplicative sections. In consequence, we can think of the quotient Lie algebra $\Gamma_{\operatorname{mult}}(V)/\textnormal{im}(\delta)$ as the space of sections of the stacky Lie algebroid presented by $V$.

\begin{rem}
We note that a broader notion, providing an alternative candidate for the concept of a Lie algebroid over a differentiable stack, has been proposed recently in \cite{AC24}. In that work, the authors introduced the notion of \emph{quasi $\mathcal{LA}$-groupoids}, arguing that these objects furnish the appropriate framework for defining Lie algebroids over differentiable stacks.
\end{rem}

	\section{The Lie algebra of basic sections}\label{sec:3}
	
In this part of the work, we introduce the notion of a basic section for a $VB$-groupoid $V$ with injective core-anchor map, and we discuss the Lie algebra structure carried by the space of basic sections in the case where $V$ is an $\mathcal{LA}$-groupoid. Firstly, this will allow us to provide a simpler model for the sections of the stacky Lie algebroids presented by $\mathcal{LA}$-groupoids with injective core-anchor map. Secondly, such a notion of basic section generalizes the notion of basic vector field over foliation groupoids introduced and studied in \cite{HS21}, but still applies to other interesting constructions where multiplicative or basic geometric structures over Lie groupoids give rise to associated $\mathcal{LA}$-groupoids.

We begin with the following standard result on $VB$-groupoids.
\begin{lemma}
	\label{lemma:quotient}		
	Let $(V\rightrightarrows E; G\rightrightarrows M)$ be a $VB$-groupoid. For any arrow $g\colon x\to y$ in $G$, we get that
	\begin{itemize}
		\item the linear isomorphism $\ker \tilde{s}_y \to \ker \tilde{s}_g$, $u_y\mapsto u_y \cdot 0_g$ induces the following linear isomorphisms
		\begin{equation*}
			\ker \tilde{s}_y \cap \ker \tilde{t}_y \to \ker\tilde{s}_g \cap \ker \tilde{t}_g, \quad
			\frac{\ker \tilde{s}_y}{\ker \tilde{s}_y\cap \ker \tilde{t}_y}\to \frac{\ker \tilde{s}_g}{\ker \tilde{s}_g\cap \ker \tilde{t}_g},
		\end{equation*}
		\item the linear isomorphism $\ker \tilde{s}_x \to \ker \tilde{t}_g$, $u_x\mapsto 0_g \cdot u_x^{-1}$ induces the following linear isomorphisms
		\begin{equation*}
			\ker \tilde{s}_x \cap \ker \tilde{t}_x \to \ker\tilde{s}_g \cap \ker \tilde{t}_g, \quad
			\frac{\ker \tilde{s}_x}{\ker \tilde{s}_x\cap \ker \tilde{t}_x}\to \frac{\ker \tilde{s}_g}{\ker \tilde{s}_g\cap \ker \tilde{t}_g}.
		\end{equation*}
	\end{itemize}

	In particular, the core-anchor map of $V$ is injective if and only if $\ker \tilde{s}_g\cap \ker \tilde{t}_g =\{0_g\}$ for all $g\in G$.
\end{lemma}

It follows from Lemma \ref{lemma:quotient} that the core-anchor map $\partial$ of a $VB$-groupoid $V$ is injective if and only if $\ker \tilde{s}\cap \ker \tilde{t}=0$, and, in this case, $\ker \tilde{s} +\ker \tilde{t} = \ker \tilde{s} \oplus \ker \tilde{t} $ is a vector subbundle of $V$. This turns out to be a Lie subalgebroid of $V$, provided that $V$ is an $\mathcal{LA}$-groupoid.
	
From now on we only consider $VB$-groupoid (resp. $\mathcal{LA}$-groupoids) $(V\rightrightarrows E; G\rightrightarrows M)$ with injective core-anchor map $\partial \colon C\to E$. In this situation, we may identify $C$ with $\operatorname{im}(\partial)$ that is a vector subbundle (resp. Lie subalgebroid) of $E$. Notice that the injectivity of the core-anchor is a Morita-invariant property, as recalled below.
\begin{rem}
	\label{rem:me_trivial_core}
	Let $F$ be a $VB$-Morita map between $VB$-groupoids $V$ and $V'$. If the core-anchor map of either $V$ or $V'$ is injective, then so is the core-anchor map of the other. This is because, by Theorem 3.5 in \cite{dHO20}, $F$ determines a pointwise quasi-isomorphism between the core complexes of $V$ and $V'$:
	\begin{equation}
		\label{eq:quis}
		\begin{tikzcd}
			0 \arrow[r] & C \arrow[r, "\partial"] \arrow[d, "F"'] & E \arrow[r] \arrow[d, "F"] & 0 \\
			0 \arrow[r] & C' \arrow[r, "\partial '"'] & E' \arrow[r] & 0
		\end{tikzcd}.
	\end{equation}

	Furthermore, any $VB$-groupoid $V$ with injective core-anchor map is Morita equivalent to one with trivial core. Indeed, if the core-anchor map $\partial \colon C\to E$ of $V$ is injective then the quotient $E/C$ supports a representation of $G$ and so can be promoted to a trivial core $VB$-groupoid $V'$. Finally, the projection $E\to E/C$ determines a $VB$-Morita map from $V$ to $V'$ (see \cite{GSM17} for additional details, where $VB$-groupoids with injective core-anchor map form a special kind of $VB$-groupoids having type $0$).
	
    Of course, if $F$ is an $\mathcal{LA}$-Morita map between $\mathcal{LA}$-groupoids $V$ and $V'$, it still holds that the injectivity of the core-anchor map in one implies the injectivity in the other. Nevertheless, unlike the case of $VB$-groupoids, an $\mathcal{LA}$-groupoid $V$ whose core-anchor map is injective is not necessarily Morita equivalent to an $\mathcal{LA}$-groupoid whose core is trivial, as $E/C$ is not a Lie algebroid in general. For instance, if $G\rightrightarrows M$ is a foliation groupoid integrating the tangent distribution $T\mathcal{F}\to M$ to a regular foliation $\mathcal{F}$ on $M$, then the core-anchor map of the tangent $VB$-groupoid $TG$ is just the inclusion $T\mathcal{F}\hookrightarrow TM$, but $\Gamma(T\mathcal{F})$ is not an ideal of $\mathfrak{X}(M)$, unless $T\mathcal{F}=TM$.
\end{rem}
	
Let $V$ be a $VB$-groupoid whose core-anchor is injective. One can consider the normal vector bundles 
\begin{equation*}
	N_1=\frac{V}{\ker(\tilde{s})+ \ker(\tilde{t})} \quad \text{and} \quad N_0=\frac{E}{C},
\end{equation*} 
over $G$ and $M$, respectively. Additionally, there are fiberwise surjective $VB$-morphisms $p^R\colon V \to s^\ast E$ and $p^L\colon V \to t^\ast E$ covering $\operatorname{id}_G$, which are respectively given by
\begin{equation}
	\label{eq:projections}
   u_g\mapsto (g, \tilde{s}(u_g)) \quad\textnormal{and}\quad u_g\mapsto (g, \tilde{t}(u_g)).
\end{equation}

The following result is a consequence of Lemma \ref{lemma:quotient}, although we provide a less computational proof.

	\begin{lemma}
		\label{lemma:pullback}
		The $VB$-morphisms $p^R$ and $p^L$ defined in Equation \eqref{eq:projections} determine $VB$-isomorphisms between $N_1$ and $s^\ast N_0$ as well as $N_1$ and $t^\ast N_0$, respectively.
	\end{lemma}
	\begin{proof}
		
		It is simple to see that the $VB$-morphisms $p^R$ and $p^L$ satisfy $\ker p^R=\ker \tilde{s}\subset \ker\tilde{s} \oplus \ker \tilde{t}$ and $\ker p^L= \ker \tilde{t}\subset \ker \tilde{s}\oplus \ker \tilde{t}$. Moreover, they also verify that
		\begin{equation*}
			p^R(\ker \tilde{s}\oplus \ker \tilde{t}) = \tilde{s}(\ker\tilde{t}) = \tilde{s} (\tilde{i}(\ker\tilde{s}))= \tilde{t}(\ker \tilde{s})= \partial(C),
		\end{equation*} 
		and
		\begin{equation*}
			p^L(\ker\tilde{s}\oplus \ker\tilde{t})= \tilde{t}(\ker\tilde{s})= \partial(C).
		\end{equation*}
	
		Consequently, $p^R$ and $p^L$ induce $VB$-isomorphisms $N_1\to s^\ast N_0$ and $N_1\to t^\ast N_0$, both covering $\operatorname{id}_G$:
		\begin{equation*}
			\begin{tikzcd}
				V \arrow[r, "p^R"] \arrow[d] & s^\ast E \arrow[d] & & V \arrow[r, "p^L"] \arrow[d] & t^\ast E \arrow[d] \\
				N_1 \arrow[r, "\bar{p^R}", dashed, '] & s^\ast N_0 && N_1 \arrow[r, "\bar{p^L}", dashed, '] & t^\ast N_0
			\end{tikzcd}.\qedhere
		\end{equation*}
	\end{proof}

We can now introduce our main concept. Let $V$ be a $VB$-groupoid whose core-anchor map is injective.

\begin{definition}
 A \emph{basic section} of $V$ is a section $\lambda\in \Gamma(N_0)$ verifying $\tilde{s}^\ast\lambda = \tilde{t}^\ast \lambda \in \Gamma(N_1)$. The vector space of all basic sections of $V$ is denoted by $\Gamma_{\operatorname{bas}}(V)$.
\end{definition}

Notice that the section $\lambda_1=\tilde{s}^\ast\lambda = \tilde{t}^\ast \lambda$ is invariant with respect to the induced inversion $\tilde{i}^\ast\colon \Gamma(N_1)\to \Gamma(N_1)$. 

Let us consider the case where $V$ is an $\mathcal{LA}$-groupoid. We want to prove that under this assumption $\Gamma_{\operatorname{bas}}(V)$  naturally inherits a Lie algebra structure. Applying the basic facts developed in Appendix \ref{appendix} to the case where $A=E$ and $B=C$, one can also consider the Bott representation $\nabla \colon \Gamma(C)\times \Gamma(N_0) \to \Gamma(N_0)$ of $C$ on $N_0$ defined by sending $(X, \overline{Y})\mapsto \overline{[X,Y]_E}$ as well as the Lie algebra of flat sections $\Gamma_0(E)$ of $E$, which is given by
		\[
			\Gamma_0(E):=\left\{\overline{Y}\in \Gamma(N_0) \, | \, \nabla_X \overline{Y}=0, \, \text{for all } X\in \Gamma(C)\right\} \cong \frac{N(\Gamma(C))}{\Gamma(C)},
		\]
		where $N(\Gamma(C))$ stands for the normalizer of $\Gamma(C)$ in the Lie algebra $\Gamma(E)$.

\begin{rem}\label{Bott_Connection_Naturality}
Since the core-anchor map is injective, the vector bundle $\ker\tilde{s}+ \ker\tilde{t}$ is a Lie subalgebroid of $V$, and we can similarly consider the Bott representation $\widetilde{\nabla}\colon \Gamma(\ker(\tilde{s})+ \ker(\tilde{t}))\times \Gamma(N_1) \to \Gamma(N_1)$ of $\ker \tilde{s} + \ker\tilde{t}$ on $N_1$. It follows that the Bott representations $\nabla$ and $\widetilde{\nabla}$ satisfy the following naturality condition. Once again, we can apply the basic facts developed in Appendix \ref{appendix} to the case where the Lie algebroids are $A_1=V$ and $A_0=E$, the fiberwise surjective Lie algebroid morphism $\Phi$ is either $\tilde{s}$ or $\tilde{t}\colon V \to E$ and the Lie subalgebroids are $B_0=C$ and necessarily $B_1=\ker\tilde{s}+ \ker\tilde{t}$. Therefore, Diagram \eqref{diag:Bott_connections} gives rise to the following two commutative diagrams:
		\begin{equation}
			\label{diag:bott}
			\begin{tikzcd}
				\ker \tilde{s} + \ker \tilde{t} \arrow[r, "\tilde{s}"] \arrow[d, "\widetilde{\nabla}"'] & C \arrow[d, "\nabla"] & \ker \tilde{s} + \ker \tilde{t} \arrow[r, "\tilde{t}"] \arrow[d, "\widetilde{\nabla}"'] & C \arrow[d, "\nabla"] \\
				DN_1 \arrow[r, "D\tilde{s}"'] & DN_0 & DN_1 \arrow[r, "D\tilde{t}"'] & DN_0
			\end{tikzcd},
		\end{equation}
		where $DN_1$ and $DN_0$ stand for the Atiyah algebroids associated with $N_1$ and $N_0$, respectively. Recall that the Atiyah algebroid associated with a vector bundle $W$ is the Lie algebroid $DW$ whose sections are derivations of $W$, the anchor is given by the corresponding symbol map and the bracket is the commutator of derivations.
	\end{rem}

The previous facts enable us to show the following key result.

	\begin{prop}
		\label{prop:subalgebra}
		Let $V$ be an $\mathcal{LA}$-groupoid whose core-anchor map is injective. Then, the vector space of basic sections $\Gamma_{\operatorname{bas}}(V)$ of $V$ is a Lie subalgebra of $\Gamma_0(E)$.
	\end{prop}
	\begin{proof}
		Let us pick $\overline{X}\in \Gamma_{\operatorname{bas}}(V)$ and $Y\in \Gamma(C)$. The right-invariant section $\overrightarrow{Y}$ generated by $Y$ is a section of $\ker(\tilde{s})+ \ker(\tilde{t})$ such that
		\[
			\tilde{s}\circ \overrightarrow{Y} = 0^C \circ s, \quad \text{and} \quad \tilde{t} \circ \overrightarrow{Y} = t \circ Y,
		\] 
		where $0^C\in \Gamma(C)$ is the zero section and $Y\equiv\partial(Y)\in \Gamma(C)$. Thus, by Remark \ref{Bott_Connection_Naturality} it follows that
		\[
			t^\ast\left( \nabla_{Y} \overline{X}\right) = \widetilde{\nabla}_{\overrightarrow{Y}} t^\ast \overline{X} = \widetilde{\nabla}_{\overrightarrow{Y}} s^\ast \overline{X} =0.
		\]
		
		Hence, $\Gamma_{\operatorname{bas}}(V)$ is a vector subspace of $\Gamma_0(E)$. We can use the facts explained in Appendix \ref{appendix} for $A_1= V$ and $A_0=E$, the Lie algebroid morphism $\Phi$ is either $\tilde{s}$ or $\tilde{t}$, and the Lie subalgebroids are $B_0=C$ and necessarily $B_1=\ker \tilde{s}+ \ker \tilde{t}$. In consequence, one obtains that the linear maps $\tilde{s}^\ast, \tilde{t}^\ast \colon \Gamma(N_0)\to \Gamma(N_1)$, which are given by the pullbacks along the fiberwise invertible $VB$-morphisms $\tilde{s}, \tilde{t}\colon N_1\to N_0$, send $\Gamma_0(E)$ to $\Gamma_0(V)$, and, moreover, the induced maps $\tilde{s}^\ast, \tilde{t}^\ast\colon \Gamma_0(E)\to \Gamma_0(V)$ are Lie algebra morphisms. Thus, given $\overline{X}, \overline{Y}\in\Gamma_0(E)$, with $\tilde{s}^\ast\overline{X}= \tilde{t}^\ast \overline{X}$ and $\tilde{s}^\ast\overline{Y}= \tilde{t}^\ast \overline{Y}$, we can compute
		\[
			\tilde{s}^\ast([\overline{X}, \overline{Y}])= [\tilde{s}^\ast \overline{X}, \tilde{s}^\ast \overline{Y}]= [\tilde{t}^\ast \overline{X}, \tilde{t}^\ast \overline{Y}]= \tilde{t}^\ast [\overline{X}, \overline{Y}],
		\]
		showing that $\Gamma_{\operatorname{bas}}(V)\subset \Gamma_0(E)$ is a Lie subalgebra.
	\end{proof}
	
A couple of observations come in order.

	\begin{rem} \label{rem:trivial_and_iso_core}
	A particular case of our situation is provided when the core $C$ of the $\mathcal{LA}$-groupoid $V$ is zero. In this case, $E$ supports a representation of $G$, the vector bundle $N_0$ coincides with $E$ and the basic sections $\Gamma_{\operatorname{bas}}(V)$ are the $G$-invariant sections of $E$ that we denote by 
	\[
		\Gamma(E)^G:=\{e\in \Gamma(E) \, | \, g. e_{s(g)}= e_{t(g)} , \text{ for all } g\in G\}. 
	\]
	
	They actually agree with the multiplicative sections of $V$, as proven in \cite{OW19}. Besides, the space $\Gamma_{\operatorname{bas}}(V)$ of basic sections of an $\mathcal{LA}$-groupoid with injective core-anchor map coincides with the space of $G$-invariant sections of $E/C$. Indeed, a section $\overline{X}\in \Gamma(E/C)$ is basic with respect to $V$ if and only if, for any $g\in G$, we have that $\tilde{s}^{-1}_g(\overline{X}_{s(g)})$ equals $\tilde{t}^{-1}_g(\overline{X}_{t(g)})$. In other words, the latter happens if and only if $\tilde{t}(\tilde{s}_g^{-1}(\overline{X}_{s(g)}))= X_{t(g)}$, where the the left-hand side expresses the action of $g$ on $\overline{X}_{s(g)}$.

	Another extreme situation is given when the core-anchor map of $V$ is an isomorphism. Here the normal bundle $N_0$ is trivial, so there are no non-trivial basic sections on $V$. Moreover, the multiplicative sections on $V$ are codified by the trivial Lie $2$-algebra $
			\begin{tikzcd}
				\Gamma(E)\arrow[r, "\operatorname{id}"] & \Gamma(E).
			\end{tikzcd}
$ See Section 3.5 in \cite{OW19} for details.
	\end{rem}

\begin{rem}
It is well-known that $\mathcal{LA}$-groupoids can alternatively be described either as degree $1$ $Q$-groupoids (see \cite{Me09}) or as $PVB$-groupoids (see \cite{BCdH16,Ma99}). Consequently, the Lie algebra of basic sections can also be described in terms of either of these two frameworks.
\end{rem}

We want to check now that the notion of basic section that we have just introduced is Morita invariant. Let us start by considering the case where $V$ and $V'$ are just $VB$-groupoids, so that $\Gamma_{\operatorname{bas}}(V)$ and $\Gamma_{\operatorname{bas}}(V')$ are just vector spaces. If $F\colon V\to V'$ is a $VB$-Morita map, then it holds that the cochain map \eqref{eq:quis} is a pointwise quasi-isomorphism between the core-complexes of $V$ and $V'$, see Theorem 3.5 in \cite{dHO20}. On the one side, the induced map $F\colon N_0 \to N_0'$, where $ N_0'=E'/C'$, is a fiberwise invertible $VB$-morphism covering $f\colon M\to M'$. This allows us to consider the pullback of sections $F^\ast \colon \Gamma(N_0')\to \Gamma(N_0)$ through $F$. On the other side, the map $F$ given by
	\[
		N_1:=\frac{V}{\ker \tilde{s}+ \ker \tilde{t}} \to \frac{V'}{\ker \tilde{s}' + \ker \tilde{t}'}=:N_1', \quad \overline{v_g}\mapsto \overline{F(v_g)},
	\]
	is also a fiberwise invertible $VB$-morphism, covering $f\colon G\to G'$. Indeed, $F\colon N_1\to N_1'$ fits into the following commutative diagram
	\begin{equation*}
		\label{diag:VB}
		\begin{tikzcd}
			N_1\arrow[r, "F"] \arrow[d, "\tilde{s}"'] & N_1 \arrow[d, "\tilde{s'}"] \\
			N_0 \arrow[r, "F"']& N_0'
		\end{tikzcd},
	\end{equation*}
	where the other three $VB$-morphisms are fiberwise invertible. By pulling back along $F\colon N_1 \to N_1'$ we obtain the linear map $F^\ast \colon \Gamma(N_1')\to \Gamma(N_1)$. It follows that the two pullback maps $F^\ast \colon \Gamma(N_0')\to \Gamma(N_0)$ and $F^\ast \colon \Gamma(N_1')\to \Gamma(N_1)$ fit into the following commutative diagrams
	\[
	\begin{tikzcd}
		\Gamma(N_1) & \Gamma(N_0) \arrow[l, "\tilde{s}^\ast"'] & \Gamma(N_1) & \Gamma(N_0) \arrow[l, "\tilde{t}^\ast"'] \\
		\Gamma(N_1') \arrow[u, "F^\ast"] &\Gamma(N_0') \arrow[l, "\tilde{s'}^\ast"] \arrow[u, "F^\ast"'] & \Gamma(N_1') \arrow[u, "F^\ast"] &\Gamma(N_0') \arrow[l, "\tilde{t'}^\ast"] \arrow[u, "F^\ast"'] 
	\end{tikzcd}.
	\]
	
	Hence, the linear map $F^\ast\colon \Gamma(N_0')\to \Gamma(N_0)$ sends $\Gamma_{\operatorname{bas}}(V')$ into $\Gamma_{\operatorname{bas}}(V)$.

\begin{rem}
	\label{rem:Lie_morphism}
	If $V$ and $V'$ are $\mathcal{LA}$-groupoid morphisms whose core-anchor maps are injective and $F\colon V\to V'$ is an $\mathcal{LA}$-Morita map, then the linear map $F^\ast\colon \Gamma_{\operatorname{bas}}(V')\to \Gamma_{\operatorname{bas}}(V)$ is a Lie algebra morphism. Indeed, let $X', Y'\in \Gamma(E')$ such that $\overline{X'}, \overline{Y'}\in \Gamma_{\operatorname{bas}}(V')$ and $x\in M$. Pick $X,Y\in \Gamma(E)$ such that $F\circ X = X'\circ f$ and $F\circ Y= Y'\circ f$ around $x$ (compare Proposition 7.3 in \cite{OW19}). Then, we get that
	\[
		F^\ast([\overline{X'}, \overline{Y'}])_x= F_x^{-1}(\overline{[X',Y']_{f(x)}}) = \overline{[X, Y]}_x = [\overline{X}, \overline{Y}]_x= [F^\ast \overline{X'}, F^\ast \overline{Y'}]_x. \qedhere
	\]
\end{rem}
	\color{black}
	\begin{lemma}
		\label{lemma:natural_transformation}
		Let $V$ and $V'$ be two $VB$-groupoids with injective core-anchor maps and let $F,K\colon V\to V'$ be two $VB$-Morita maps covering the same map $f\colon G\to G'$. If $F$ and $K$ are related by a linear natural transformation, then the linear maps $F^\ast$ and $K^\ast\colon \Gamma_{\operatorname{bas}}(V')\to \Gamma_{\operatorname{bas}}(V)$ agree.
	\end{lemma}
	\begin{proof}
		Suppose that $\alpha\colon E\to V'$ is a linear natural transformation from $F$ to $K$. By Theorem 3.8 in \cite{MTV24} we get a pointwise homotopy 
		\[
			\begin{tikzcd}
				0 \arrow[r] & C \arrow[r, "\partial"] \arrow[d, shift left=0.5ex, "F"] \arrow[d, shift right=0.5ex, "K"'] & E \arrow[dl] \arrow[r] \arrow[d, shift left=0.5ex, "F"] \arrow[d, shift right=0.5ex, "K"'] & 0\\
				0 \arrow[r] & C' \arrow[r, "\partial'"'] & E' \arrow[r] & 0,
			\end{tikzcd}
		\]
		between the cochain maps determined by $F$ and $K$, so that the claim follows.
	\end{proof}
	
As an easy consequence of Remark \ref{rem:Lie_morphism} and Lemma \ref{lemma:natural_transformation} one deduces that:

	\begin{coroll}
		Let $V$ and $V'$ be two $\mathcal{LA}$-groupoids with injective core-anchor maps and let $F,K\colon V\to V'$ be two $\mathcal{LA}$-Morita maps covering the same map $f\colon G\to G'$. If $F$ and $K$ are related by a linear natural transformation, then the Lie algebra morphisms $F^\ast$ and $K^\ast\colon \Gamma_{\operatorname{bas}}(V')\to \Gamma_{\operatorname{bas}}(V)$ agree.
	\end{coroll}

More importantly, in view of Remark \ref{rem:me_trivial_core}, we obtain the following result.

	\begin{prop}
		\label{prop:morita_invariance}
		Let $(F,f)\colon (V\rightrightarrows E; G\rightrightarrows M)\to (V'\rightrightarrows E'; G'\rightrightarrows M')$ be a $VB$-Morita map between two $VB$-groupoids, one of which has injective core-anchor map. Then, the pullback map $F^\ast \colon\Gamma_{\operatorname{bas}}(V') \to \Gamma_{\operatorname{bas}}(V)$ is a linear isomorphism.
	\end{prop}
	\begin{proof}

	 We consider first the case when  $f=\operatorname{id}_G$. In this scenario, by Proposition 6.2 in \cite{dHO20}, there exists a $VB$-groupoid morphism $K\colon V'\to V$ covering $\operatorname{id}_G$ and linear natural transformations $K\circ F \Rightarrow \operatorname{id}_V$ and $F\circ K \Rightarrow \operatorname{id}_{V'}$. By Lemma \ref{lemma:natural_transformation} we obtain that $(K\circ F)^\ast= F^\ast\circ K^\ast$ agrees with the identity $\operatorname{id}_{\Gamma_{\operatorname{bas}}(V)}$ and that $(F\circ K)^\ast= K^\ast\circ F^\ast$ agrees with the identity $\operatorname{id}_{\Gamma_{\operatorname{bas}}(V')}$, so that $F^\ast$ is an isomorphism.
		
	 If we consider the situation for general $f\colon G\to G'$, then we may look at the pullback $VB$-groupoid $f^\ast V'$ over $G$ which is given by $G\times_{G'} V' \rightrightarrows M \times_{M'} E'$. The core complex of $f^\ast V'$ is
		\[
			\begin{tikzcd}
				0 \arrow[r] & f^\ast C' \arrow[r] & f^\ast E' \arrow[r] & 0,
			\end{tikzcd}
		\]
		where the map $f^\ast C'\to f^\ast E'$ is essentially given by $\partial'$ and so it is injective. In particular:
		\[
			N:=f^\ast E'/f^\ast C' \cong f^\ast\left(E'/C'\right) = f^\ast N_0'.
		\]
		
		This implies that $\Gamma(N)= \Gamma(N_0)$ and therefore $\Gamma_{\operatorname{bas}}(f^\ast V') = \Gamma_{\operatorname{bas}}(V')$. The $VB$-groupoid $f^\ast V'$ comes with a $VB$-groupoid morphism $\mathsf{F}\colon V \to f^\ast V'$ covering $\operatorname{id}_G$ that is defined by the expression
		\[
			\mathsf{F}(v) = (g, F(v))\in G\times_{G'} V',\quad v\in V_g.
		\] 
		
		Hence, by Theorem 3.5 in \cite{dHO20} it holds that $\mathsf{F}$ is a $VB$-Morita map covering the identity, thus obtaining that $\mathsf{F}^\ast \colon \Gamma_{\operatorname{bas}}(f^\ast V')\to \Gamma_{\operatorname{bas}}(V)$ is an isomorphism. However, because of the equivalence $\Gamma_{\operatorname{bas}}(f^\ast V')= \Gamma_{\operatorname{bas}}(V')$ we conclude that the maps $\mathsf{F}^\ast$ and $F^\ast$ coincide and so the result follows, as desired.
	\end{proof}

It is clear that if $(F,f)\colon (V\rightrightarrows E; G\rightrightarrows M)\to (V'\rightrightarrows E'; G'\rightrightarrows M')$ is instead an $\mathcal{LA}$-Morita map between two $\mathcal{LA}$-groupoids, one of which has injective core-anchor map, then $F^\ast \colon\Gamma_{\operatorname{bas}}(V') \to \Gamma_{\operatorname{bas}}(V)$ is a Lie algebra isomorphism. Even more, we can now easily deduce the following result.

\begin{coroll}
If $V$ and $V'$ are two Morita equivalent $\mathcal{LA}$-groupoids, then $\Gamma_{\operatorname{bas}}(V)\cong \Gamma_{\operatorname{bas}}(V')$ as Lie algebras.
\end{coroll}
	
In consequence, the Lie algebra of basic sections $\Gamma_{\operatorname{bas}}(V)$ gives rise to a model for the sections of the stacky Lie algebroid presented by the $\mathcal{LA}$-groupoid $V$, provided that its core-anchor map is injective.

As alluded to previously, the Lie 2-algebra of multiplicative sections of an $\mathcal{LA}$-groupoid also provides a model for the sections of the corresponding stacky Lie algebroid \cite{OW19}, and these two models turns out to be equivalent, as the following result shows.

	\begin{prop}
		Let $V$ be an $\mathcal{LA}$-groupoid whose core-anchor map is injective. Then, the  map 
		\[
			\Psi\colon \frac{\Gamma_{\operatorname{mult}}(V)}{\textnormal{im}(\delta)} \to \Gamma_{\operatorname{bas}}(V), \quad \overline{(\lambda, e)}\mapsto \overline{e}
		\]
		is a Lie algebra isomorphism.
	\end{prop}
	\begin{proof}
	From the definition of the Lie algebra structures on $\Gamma_{\operatorname{mult}}(V)/\textnormal{im}(\delta)$ and $\Gamma_{\operatorname{bas}}(V)$, it easily follows that $\Psi$ is a Lie algebra morphism. On the one hand, by Remark \ref{rem:trivial_and_iso_core}, we know that $\Gamma_{\operatorname{bas}}(V)$ agrees with $\Gamma(E/C)^G$, as vector spaces. The isomorphism is just the identity of $\Gamma(E/C)$, noting that a section $\overline{X}\in \Gamma(E/C)$ is basic with respect to $V$ if and only if $\overline{X}$ is invariant with respect to the representation of $G$ described by the $VB$-groupoid $V'$. On the other hand, by Remark \ref{rem:me_trivial_core}, $V$ is Morita equivalent (at least as a $VB$-groupoid) to the trivial core $VB$-groupoid $V'$ with side bundle $E/C$. But, Remark 3.2 in \cite{OW19} says that the vector space of multiplicative sections of $V'$ is also given by $\Gamma(E/C)^G$. Therefore, the quotient vector space $\Gamma_{\operatorname{mult}}(V)/\textnormal{im}(\delta)$ is isomorphic to $\Gamma(E/C)^G$, compare Corollary 7.2 in \cite{OW19}, and the isomorphism is given exactly by $\Psi$. This completes the proof.
	\end{proof}
	
	\section{Examples}\label{sec:4}
	In this final section we exhibit some interesting examples where our results apply.
	
	\begin{example}[Basic vector fields]
		A \emph{foliation groupoid} is a Lie groupoid $G\rightrightarrows M$ whose anchor map $\rho\colon A \to TM$ is injective. They naturally show up as integrations of regular foliations. The tangent Lie groupoid $TG\rightrightarrows TM$ is an $\mathcal{LA}$-groupoid over $G\rightrightarrows M$ whose core-anchor map agrees with $\rho$ and so it is injective. Considering basic sections of the tangent $\mathcal{LA}$-groupoid one recovers the notion of \emph{basic vector field} on foliation groupoids. These were studied and exploited in the work \cite{HS21} which turned out to be a source of inspiration for our generalization. First, any Lie groupoid admitting a $0$-shifted symplectic structure is a foliation groupoid. In particular, basic vector fields are used in \cite{HS21} to get a Morita invariant reduction of Lie groupoids equipped with $0$-shifted symplectic structures under Hamiltonian actions of foliation Lie $2$-groups. Second, if we endow our foliation groupoid with a Riemannian 2-metric \cite{dHF18} then the idea of basic Killing vector field gives rise to a notion of Killing vector field over the corresponding Riemannian stack \cite{HCV23}. 
	\end{example}
	
	\begin{example}[Basic derivations] \label{ex:basic_derivation}
		Our initial motivation for this work was trying to define what would be a basic derivation of a line bundle groupoid (line bundle internal to the category of Lie groupoids), an object that is closely related to the notion of $0$-shifted contact structure recently introduced in \cite{MTV24}. If $E$ is a representation of a Lie groupoid $G\rightrightarrows M$, then the action groupoid $G\ltimes E\rightrightarrows E$ is a $VB$-groupoid with trivial core over $M$. We denote such a $VB$-groupoid by $V\rightrightarrows E$. Consider the Atiyah algebroids $DV\to G$ and $DE\to M$ of $V$ and $E$, respectively. Since the core of $V$ is trivial, we can apply the Atiyah functor $D$ to the structure maps of $V$, thus obtaining the \emph{Atiyah $\mathcal{LA}$-groupoid} $(DV\rightrightarrows DE; G\rightrightarrows M)$ (compare \cite{ETV19}). When the core-anchor map of $DV$ is injective we get a model for the \emph{basic derivations} of $V$. A particular case is given when $E=L_M$ is a line bundle over $M$, and so $V=L$ is a line bundle over $G$. More importantly, a line bundle groupoid $(L\rightrightarrows L_M; G\rightrightarrows M)$ admitting a $0$-shifted contact structure forces its corresponding Atiyah $\mathcal{LA}$-groupoid $DL\rightrightarrows DL_M$ to have injective core-anchor map \cite{MTV24}. Therefore, we end up with a notion of \emph{basic derivation} of the line bundle groupoid $L$ which is expected to allow us to obtain a reduction procedure for line bundle groupoids equipped with $0$-shifted contact structures invariant under the action of foliation Lie $2$-groups, see \cite{MV}.
	\end{example}
	
	\begin{example}[Basic $1$-forms on transitive Poisson groupoids]
		\label{ex:Poisson_groupoid}
		A \emph{Poisson groupoid} is a Lie groupoid $G\rightrightarrows M$ endowed with a suitably compatible Poisson structure on $G$. In other words, it supports a Poisson bivector $\pi$ on $G$ such that $\pi^\sharp\colon T^\ast G \to TG$ is a $VB$-groupoid morphism from the cotangent $VB$-groupoid to the tangent $VB$-groupoid. Poisson groupoids appear naturally as the global counterparts of \emph{Lie bialgebroids}, see \cite{MX00}. It is known that if $G\rightrightarrows M$ is a Poisson groupoid, then the dual $A^\ast$ can be equipped with a Lie algebroid structure, where $A$ is the Lie algebroid of $G$. Actually, the cotangent $VB$-groupoid $T^\ast G\rightrightarrows A^\ast$ is an $\mathcal{LA}$-groupoid whose core-anchor map is $\rho^\ast\colon T^\ast M \to A^\ast$, the dual of the anchor $\rho$ of $A$. If $G$ is transitive then $\rho$ is surjective, so that $T^\ast G$ turns into an $\mathcal{LA}$-groupoid whose core-anchor map is injective. Transitive Poisson groupoids and transitive Lie bialgebroids were already studied in \cite{CL05}. Hence, we get a model for the \emph{basic 1-forms} on $G$, provided that it is a transitive Poisson groupoid.
	\end{example}
	
	\begin{example}[Basic $1$-jets on Jacobi groupoids]
		Similarly to the previous example, we can consider the situation where we have a Jacobi groupoid. A \emph{Jacobi groupoid} is a Lie groupoid $G\rightrightarrows M$ endowed with a suitably compatible Jacobi structure on $G$. These turn out to be the global counterparts of generalized Lie bialgebroids, consult \cite{Da16,IM03}. Adopting the line bundle point of view (see, e.g., \cite{VW20}), a Jacobi groupoid is a $VB$-groupoid $(L\rightrightarrows L_M; G\rightrightarrows M)$, where $L\to G$ and $L_M\to M$ are line bundles equipped with a Jacobi bivector $\mathsf{J}\colon \wedge^2 J^1L \to L$ that is compatible with the $VB$-groupoid structures involved. Here $J^1L$ stands for the first jet bundle of $L$ and the compatibility condition means that the sharp map $\mathsf{J}^\sharp\colon J^1L\to DL$ is a $VB$-groupoid morphism. Note that $DL\rightrightarrows DL_M$ is the Atiyah $\mathcal{LA}$-groupoid over $G\rightrightarrows M$ introduced in Example \ref{ex:basic_derivation} and $J^1L \rightrightarrows A^\dagger =A^\ast\otimes L_M$ is the $VB$-groupoid over $G\rightrightarrows M$ obtained as the dual of $DL\rightrightarrows DL_M$ tensor $L\rightrightarrows L_M$ (see Example 3.1 in \cite{MTV24} or Remark 2.1.36 in \cite{Ma25}). It is worth stressing that the Lie algebroid $A$ of $G$ is identified with the core of $DL$ and that we are using the canonical isomorphism $J^1L \cong DL^\ast \otimes L$. Furthermore, this definition is slightly different from the one that appears in \cite{IM03}, but, in the case when $L=\mathbbm{R}_G$ and $L_M=\mathbbm{R}_M$ are the trivial line bundles, the two definitions coincide. 
		
		The Jacobi structure on $G$ determines a Lie algebroid structure on $A^\dagger$, just as it occurs for the case of Poisson groupoids (see Section 5.2 in \cite{IM03} for the trivial line bundle case). Actually, $J^1L\rightrightarrows A^\dagger$ is an $\mathcal{LA}$-groupoid over $G$. Its core-anchor map $\mathcal{D}^\dagger\colon J^1L_M\to A^\dagger$ is given by the twisted dual of the core-anchor map $\mathcal{D}\colon A\to DL_M$ of $DL$. Hence, if $\mathcal{D}$ is surjective, then $J^1L$ is an $\mathcal{LA}$-groupoid with injective core-anchor map, thus obtaining a model of \emph{basic $1$-jets} on $L$.
	\end{example}
	
	\begin{rem}
		Particular cases of the last two examples are symplectic and contact groupoids (see, e.g., Section 3.2 in \cite{BSTV20} for a definition of contact groupoids that involves line bundles). In both cases we get that, if the core-anchor map of $T^\ast G$ (resp. $J^1L$) is injective, then it is also surjective and there are no non-trivial basic $1$-forms (resp. non-trivial basic $1$-jets) neither non-trivial basic vector fields (resp. non-trivial basic derivations). Indeed, a compatible symplectic structure on a Lie groupoid $G\rightrightarrows M$ is given by a compatible Poisson structure $\pi$ such that the $VB$-groupoid morphism $\pi^\sharp\colon T^\ast G\to TG$ is invertible. Then, $\pi^\sharp$ determines the following isomorphism
		\[
			\begin{tikzcd}
				0 \arrow[r] & T^\ast M \arrow[r, "\rho^\ast"] \arrow[d, "\pi^\sharp"'] & A^\ast \arrow[d, "\pi^\sharp"] \arrow[r] & 0\\
				0 \arrow[r] & A \arrow[r, "\rho"'] & TM \arrow[r] & 0,
			\end{tikzcd}
		\]
		between the core complexes of $T^\ast G$ and $TG$. Analogously, a compatible contact structure is given by a $VB$-groupoid $L\rightrightarrows L_M$ over $G$, where $L$ and $L_M$ are line bundles, together with a compatible Jacobi structure $\mathsf{J}$ such that $\mathsf{J}^\sharp\colon J^1L\to DL$ is invertible. Once again, $\mathsf{J}^\sharp$ determines an isomorphism
		\[
		\begin{tikzcd}
			0 \arrow[r] & J^1L_M \arrow[r, "\mathcal{D}^\dagger"] \arrow[d, "\mathsf{J}^\sharp"'] & A^\dagger \arrow[d, "\mathsf{J}^\sharp"] \arrow[r] & 0\\
			0 \arrow[r] & A \arrow[r, "\mathcal{D}"'] & DL_M \arrow[r] & 0
		\end{tikzcd}
		\]
		between the core complexes of $J^1L$ and $DL$. It is clear that, if $\rho$ (resp. $\mathcal{D}$) is surjective, then it is also injective. Hence $\rho$ and $\rho^\ast$ (resp. $\mathcal{D}$ and $\mathcal{D}^\dagger$) are isomorphisms. In other words, both $TG$ and $T^\ast G$ (resp. $DL$ and $J^1L$) are $\mathcal{LA}$-groupoids with an isomorphism as core-anchor map, so that the claim follows from Remark \ref{rem:trivial_and_iso_core}.
	\end{rem}

We finish this work by pointing out that there are other interesting examples of $\mathcal{LA}$-groupoids with injective core-anchor map associated to multiplicative geometric structures on Lie groupoids, meaning that we can also get a notion of basic section for those structures. For instance, \emph{multiplicative Dirac structure} give rise to examples of $\mathcal{LA}$-groupoids \cite{Or13}. Similarly, additional examples can be obtained by introducing a notion of \emph{multiplicative Dirac-Jacobi structure} after imposing some multiplicativity condition on the notion of Dirac-Jacobi structure studied in \cite{V18}.

\appendix

\section{}

\label{appendix}

The aim of this short appendix is to collect some basic facts about Lie algebroids and their morphisms which are used throughout this work. Let $A$ be a Lie algebroid over $M$ and let $B$ be a wide Lie subalgebroid of $A$. One can consider the normal bundle $A/B$ and the Bott representation $\nabla \colon \Gamma(B)\times \Gamma(A/B) \to \Gamma(A/B)$ of $B$ on $A/B$, defined by sending $(X, \overline{Y})\mapsto \overline{[X,Y]_A}$. The vector space of \emph{flat sections} $\Gamma_0(A)$ of $A$ is by definition
	\[
	\Gamma_0(A):=\left\{\overline{Y}\in \Gamma(A/B) \, | \, \nabla_X \overline{Y}=0, \, \text{for all } X\in \Gamma(B)\right\},
	\]
	and the normalizer $N(\Gamma(B))$ of $\Gamma(B)$ in the Lie algebra $\Gamma(A)$ is given by
	\[
	N(\Gamma(B))=\left\{X\in \Gamma(A) \, | \, [X,Y]_A \in \Gamma(B) \, \text{for all } Y\in \Gamma(B)\right\}.
	\]
	
	If we denote by $\widetilde{\pi}\colon N(\Gamma(B))\to \Gamma(A/B)$ the restriction to $N(\Gamma(B))$ of the quotient $C^\infty(M)$-module morphism $\pi\colon \Gamma(A)\to \Gamma(A/B)$ then we clearly get that $\ker(\widetilde{\pi})=\Gamma(B)$ and, moreover, $\operatorname{im}(\widetilde{\pi})=\Gamma_0(A)$. Indeed, pick $\overline{X}\in \Gamma(A/B)$, so that $X\in\Gamma(A)$. Thus, $X\in N(\Gamma(B))$ if and only if $[X, Y]_A\in \Gamma(B)$ for all $Y\in \Gamma(B)$, which in turn happens if and only if for any $Y\in \Gamma(B)$ we have that
	\[
	0=\overline{[X,Y]_A} = \nabla_Y \overline{X}.
	\]
	
	So, the latter holds if and only if $\overline{X}\in \Gamma_0(A)$. It follows that $\Gamma_0(A)\cong N(\Gamma(B))/\Gamma(B)$. Additionally, as consequence of the Jacobi identity for $[-,-]_A$ in $\Gamma(A)$, we can deduce that $N(\Gamma(B))$ is actually a Lie subalgebra of $\Gamma(A)$ and that $\Gamma(B)$ is an ideal in $N(\Gamma(B))$. In other words, $N(\Gamma(B))/\Gamma(B)$ is a Lie algebra and so $\Gamma_0(A)$ is a Lie algebra as well, whose structure is such that
	\[
		[\overline{X}, \overline{Y}]_{\Gamma_0(A)}=\overline{[X,Y]_A}.
	\]

	Let us now consider a fiberwise surjective Lie algebroid morphism $\Phi\colon A_1\to A_0$ covering a submersion $\varphi\colon M_1\to M_0$. Fix $B_0\subset A_0$ a wide Lie subalgebroid. Then $B_1=\Phi^{-1}(B_0)\subset A_1$ is a wide Lie subalgebroid, the induced map $\Phi\colon B_1\to B_0$, $u_x\mapsto \Phi(u_x)$, is a fiberwise surjective Lie algebroid morphism covering $\varphi\colon M_1\to M_0$, and the induced map $\overline{\Phi}\colon A_1/B_1 \to A_0/B_0$, $\overline{u_x}\mapsto \overline{\Phi(u_x)}$, is a fiberwise invertible $VB$-morphism covering $\varphi\colon M_1\to M_0$.
	
	For $i=0,1$ consider the Bott representation $\nabla^i$ of $B_i$ on $A_i/B_i$, so that $\nabla^i_{\beta_i} \overline{\alpha_i} = \overline{[\beta_i, \alpha_i]}$, for all $\beta_i\in \Gamma(B_i)$ and $\alpha_i\in \Gamma(A_i)$. One has a corresponding de Rham complex $(\Omega^\bullet(B_i, A_i/B_i), d_{\nabla^i})$ and a Lie algebra of flat sections
	\[
	\Gamma_0(A_i):=\{\overline{\alpha_i}\in \Gamma(A_i/B_i) \, | \, d_{\nabla^i}\overline{\alpha_i}=0\}\cong \frac{N(\Gamma(B_i))}{\Gamma(B_i)}.
	\]
	
	The Lie algebroid morphism $\Phi\colon B_1\to B_0$ and the fiberwise invertible $VB$-morphism $\overline{\Phi}\colon A_1/B_1\to A_0/B_0$, both covering $\varphi\colon M_1\to M_0$, allow to draw the commutative diagram of Lie algebra morphism 
	\begin{equation}\label{diag:Bott_connections}
	\begin{tikzcd}
		B_1 \arrow[r, "\Phi"] \arrow[d, "\nabla^1"'] & B_0 \arrow[d, "\nabla^0"] \\
		D(A_1/B_1) \arrow[r, "D\overline{\Phi}"'] & D(A_0/B_0)
	\end{tikzcd},
	\end{equation}
where $D(A_i/B_i)$ stands for the Atiyah algebroid of $A_i/B_i$. Equivalently, $\Phi$ and $\overline{\Phi}$ allow one to build a cochain map
	\[
	\overline{\Phi}^\ast \colon (\Omega^\bullet(B_0, A_0/B_0), d_{\nabla^0})\to (\Omega^\bullet(B_1, A_1/B_1), d_{\nabla^1}), \quad \eta \mapsto \overline{\Phi}^\ast \eta, 
	\]
	where one sets $(\overline{\Phi}^\ast \eta)_x= \overline{\Phi}_x^{-1}\circ \eta_{\varphi(x)} \circ \wedge^\bullet \Phi_x$, for all $x\in M_1$ and $\eta\in \Omega^\bullet(B_0, A_0/B_0)$. The commutativity of Diagram \ref{diag:Bott_connections} may be verified as follows. Let $b\in B_{1,x}$, with $x\in M_1$, and $\overline{X}\in \Gamma(A_0/B_0)$. Pick $Y'\in \Gamma(B_1)$ such that $Y'_x=b$, and $X'\in\Gamma(A_1)$ and $Y\in \Gamma(B_0)$ such that $\Phi\circ X'= X\circ \varphi$ and $\Phi\circ Y'= Y\circ \varphi$. Hence, we get that
	\[
		D\overline{\Phi}(\nabla^1_b)(\overline{X}) = \Phi(\nabla^1_b \overline{\Phi}^\ast \overline{X}) = \Phi(\nabla^1_b \overline{X'}) = \Phi(\overline{[Y', X']_x})=\overline{\Phi([Y', X']_x)}= \overline{[Y', X']_{\varphi(x)}} = \nabla^0_{\Phi(b)} \overline{X}.
	\] 
	
	In particular, $\overline{\Phi}^\ast$ sends $1$-cocycles to $1$-cocycles and so it immediately induces an injective linear map
	\[
	\overline{\Phi}^\ast \colon \Gamma_0(A_0)\to \Gamma_0(A_1), \quad \overline{\alpha_0}\mapsto \overline{\Phi}^\ast \overline{\alpha_0}.
	\]
	
	Actually, one can prove that the latter is a Lie algebra morphism. Indeed, given two arbitrary elements $\alpha_0, \alpha_0'\in \Gamma(A_0)$, pick $\alpha_1, \alpha_1'\in \Gamma(A_1)$ such that $\overline{\Phi}_x\circ \alpha_{1,x} =\alpha_{0,\varphi(x)}$ and $\overline{\Phi}_x\circ \alpha_{1,x}' =\alpha_{0,\varphi(x)}'$, for all $x\in M_1$. This implies, by definition of $\overline{\Phi}$, that $\overline{\Phi}^\ast (\overline{\alpha_0})=\overline{\alpha_1}$ and $\overline{\Phi}^\ast (\overline{\alpha_0'})=\overline{\alpha_1'}$. Moreover, since $\Phi\colon A_1\to A_0$ is a Lie algebroid morphism, this also implies that $\Phi_x([\alpha_1, \alpha_1']_x)= [\alpha_0, \alpha_0']_{\varphi(x)}$, for all $x\in M_1$. Thus
	\[
	(\overline{\Phi}^\ast [\overline{\alpha_0}, \overline{\alpha_0'}])_x= \overline{\Phi}_x^{-1} (\overline{[\alpha_0, \alpha_0']_{\varphi(x)}}) = \overline{[\alpha_1, \alpha_1']_x},
	\]
	for all $x\in M_1$.

	
\end{document}